\documentclass[10pt]{amsart}
\usepackage{amsfonts}
\usepackage{amsmath}
\usepackage{amssymb}
\usepackage{graphicx}
\usepackage{enumerate}
\usepackage{multicol}
\usepackage{mathrsfs}

\newtheorem{theorem}{Theorem}
\newtheorem{corollary}{Corollary}
\newtheorem{lemma}{Lemma}
\newtheorem{definition}{Definition}
\newtheorem{proposition}{Proposition}
\theoremstyle{remark}
\newtheorem{remark}{Remark}
\newtheorem{example}{Example}

\DeclareMathOperator{\real}{{\mathbb{R}}}

\DeclareMathOperator{\calO}{\mathcal{O}}
\DeclareMathOperator{\calD}{\mathcal{D}}
\DeclareMathOperator{\calV}{\mathcal{V}}
\DeclareMathOperator{\calE}{\mathcal{E}}
\DeclareMathOperator{\calF}{\mathcal{F}}
\DeclareMathOperator{\calR}{\mathcal{R}}

\DeclareMathOperator{\scrW}{\mathscr{W}}
\DeclareMathOperator{\scrE}{\mathscr{E}}

\DeclareMathOperator{\spn}{span}

\DeclareMathOperator{\rank}{rank}

\DeclareMathOperator{\SO}{SO}

\DeclareMathOperator{\GL}{GL}

\DeclareMathOperator{\so}{\mathfrak{so}}
\DeclareMathOperator{\gl}{\mathfrak{gl}}
\DeclareMathOperator{\frakg}{\mathfrak{g}}
\DeclareMathOperator{\Ad}{Ad}

\DeclareMathOperator{\Lie}{Lie}

\title[Controllability of rolling in higher dim.]{Controllability of rolling without twisting or slipping in higher dimensions}
\author[E. Grong]{Erlend Grong}

\address{Department of Mathematics, University of Bergen, Norway.}
\email{erlend.grong@math.uib.no}



\begin{document}
\maketitle

\begin{abstract}
We describe how the dynamical system of rolling two $n$-dimensional connected, oriented Riemannian manifolds $M$ and $\widehat M$ without twisting or slipping, can be lifted to a nonholonomic system of elements in the product of the oriented orthonormal frame bundles belonging to the manifolds. By considering the lifted problem and using properties of the elements in the respective principal Ehresmann connections, we obtain sufficient conditions for the local controllability of the system in terms of the curvature tensors and the sectional curvatures of the manifolds involved.
We also give some results for the particular cases when $M$ and $\widehat M$ are locally symmetric or complete.
\end{abstract}

\section{Introduction}
The rolling of two surfaces, without twisting or slipping, is a good illustration of a nonholonomic mechanical system,  whose properties are intimately connected with geometry. It has therefore received much interest, and we can mention \cite{A,BH, BM, MB} and \cite[Chapter 24]{AS} as examples of research produced in this area.
In particular, the treatment of rolling in \cite{AS,BH} was done by formulating it as an intrinsic problem, independent of the imbedding of the surfaces into Euclidean space.

The generalization of this concept to that of an $n$-dimensional manifold rolling without twisting or slipping on the $n$-dimensional Euclidean space, is well known (see e.g. \cite[p. 268]{IW}, \cite[Chapter 2.1]{Hsu}). It is usually formulated intrinsically, in terms of frame bundles, and is an important tool in stochastic calculus on manifolds. A definition for two arbitrary $n$-dimensional manifolds rolling on each other without twisting or slipping, first appeared in \cite[App. B]{Sharpe}, however, this only dealt with manifolds imbedded into Euclidean space. An intrinsic definition for rolling of higher dimensional manifolds, that connected the definition in \cite{Sharpe} with the intrinsic approach in \cite{AS,BH}, was presented in \cite{GGMS}. Apart from appearing as mechanical systems, rolling of higher dimensional manifolds can be also used as a tool in interpolation theory. For demonstration of the ``rolling and wrapping''-technique, we refer to \cite{HS}. See also \cite{Shen} for an example where this is applied in robot motion planning.

For the rolling of two 2-dimensional manifolds, there is a beautiful correspondence between the degree of control and the geometry of the manifolds \cite{AS,BH}. Essentially, we have complete control over our dynamical system if the respective Gaussian curvatures $M$ and $\widehat M$ do not coincide.
Controllability in higher dimensions has been addressed in some special cases \cite{JZ,Zimm}.
The first general result on controllability in higher dimension is presented in \cite{Chitour}, where it is shown that the curvature tensors determine the brackets of the distribution obtained from the constraints given by neither twisting, nor slipping. It is also proved that being able to find a rolling to an arbitrary close configuration, connecting the same two points, is a sufficient condition for complete controllability.

The objective of this paper will be to describe the connection between geometry and controllability for rolling of higher dimensional manifolds. We do this by connecting the earlier mentioned viewpoint from stochastic calculus with the one presented in \cite{GGMS}.

This paper is organized as follows. In Section 2, we state the intrinsic definition of rolling. We present some of the theory of frame bundles, and develop our notation there. We end this section by showing how we can lift our problem to the oriented orthonormal frame bundles of the involved manifolds. We continue in section Section 3, by doing computations on the lifted problem. By using properties of the sections in the Ehresmann connections, we obtain formulas for computation of the brackets of the rolling distribution. In Section 4 we project the results back to our configuration space of relative positions of the manifolds. Section 4.2 consist of conditions for controllability in terms of the Riemann curvature tensor and the sectional curvature of the manifolds involved. We end this section with some examples. Section 5 focuses on results concerning the rolling of locally symmetric and complete manifolds. Section 6 contains a brief comment on how to generalize the concept of rolling without twisting or slipping to manifolds with an affine connection, and why the results presented here also holds for a rolling of manifolds with a torsion free affine connection.

The author would like to express his gratitude to Mauricio Godoy Molina for many fruitful discussions concerning this subject.

\section{Intrinsic definition of rolling and its relations to frame bundles}
\subsection{Intrinsic definition of rolling without twisting or slipping} \label{definition}
Throughout this paper, $M$ and $\widehat M$ will denote connected, oriented, $n$-dimensional Riemannian manifolds. Since in the special case $n=1$, the conditions of rolling without twisting or slipping become holonomic (see \cite{GGMS}), we will always assume that $n \geq 2$. We adopt the convention to equip objects (points, projection, etc.) related to $\widehat{M}$ with a hat (\^{}). Objects related to both of them are usually denoted by a bar (\={}), while objects connected to $M$ are not given any special distinction. The exception to this rule is the Riemannian metric and the affine Levi-Civita connection which are respectively denoted by $\langle \cdot , \cdot \rangle$ and $\nabla$ on both $M$ and $\widehat{M}$. The context will make it clear which manifold these objects are related to. For a vector field $X$ on $M$, we will write $X|_m$ rather than $X(m)$, and we will use similar notation for other sections of bundles.

For any pair of oriented inner product spaces $V$ and $\widehat V$, we let $\SO(V,\widehat V)$ denote the space of all orientation preserving linear isometries from $V$ to $\widehat V$. This allows us to define the $\SO(n)$-fiber bundle~$Q$ over $M \times \widehat{M}$ by
$$Q = \left\{ q \in \SO\left(T_mM, T_{\widehat{m}} \widehat{M}\right) : \, m \in M, \widehat{m} \in \widehat{M}\right\}.$$
We can be sure that this fiber bundle is principal in the case when $n=2$, but not in general. The space $Q$ represents all configurations or relative positions of $M$ and $\widehat{M}$, so that the two manifolds lie tangent to each other at some pair of points. The isometry $q:T_m M \to T_{\widehat{m}}\widehat{M},$ represents a configuration where $M$ at $m$ lies tangent to $\widehat{M}$ at $\widehat{m}$. The relative positioning of their tangent spaces is given by how $q$ maps $T_mM$ into $T_{\widehat{m}}\widehat{M}$. A rolling then becomes a curve in the space of these configurations.

\begin{definition} \label{defrolling}
Let $\pi$ and $\widehat \pi$ denote the respective natural projections from $Q$ to $M$ and $\widehat M$. A rolling without twisting or slipping is an absolutely continuous curve
$$q:[0, \tau] \to Q, \quad \text{with} \quad m(t) := \pi(q(t)), \quad \widehat m(t) := \widehat \pi(q(t)),$$
satisfying the following conditions: \\
\begin{tabular}{rl} No slip condition: & $\dot{\widehat m}(t) = q(t) \dot m(t)$ for almost every $t$, \\ 
No twist condition: & an arbitrary vector field $X(t)$ is parallel along $m(t)$ \\
& if and only if $q(t)X(t)$ is parallel along $\widehat m(t)$. \end{tabular}
\end{definition}
From now on we will mostly refer to a rolling $q(t)$ without twisting or slipping as simply a rolling.

These two conditions can be described in terms of a distribution $D$ of rank $n$ on $Q$. Consider the problem of finding the rolling $q(t)$ between two different configurations $q_0$ and $q_1$ such that $m(t)$ (and thereby also $\widehat m(t)$) has minimal length. Here, by a distribution on $Q$, we mean a sub-bundle of $TQ$. This type of curves can be viewed as optimal curves in an input-linear drift free optimal control problem or length minimizers in a sub-Riemannian manifold. We will describe this structure in more detail in Section \ref{rollingdistribution}. First, however, we will review some facts about connections and frame bundles.

\subsection{From principal bundles to oriented orthonormal frame bundles}
Consider a general principal $G$-bundle $\tau:P \to M$, where the Lie group $G$ acts on the right. We call the sub-bundle $\calV := \ker \tau_*$ of $TP$ {\it the vertical space of} $P$. If $\frakg$ is the Lie algebra of $G$, then for any element $A \in \frakg$, we have a vector field $\sigma(A)$ defined by
\begin{equation} \label{VAcor} \sigma(A)|_p\phi = \left. \frac{d}{dt} \right|_{t=0} \phi\left(p \cdot \exp_G(tA)\right), \qquad \text{for any } p \in P, \phi \in C^\infty(P).\end{equation}
Here, $\exp_G: \frakg \to G$ is the group exponential. We remark that $\sigma(A)$ is a section of $\calV$ and for any $p \in P$, the map $\frakg \to \calV_p, A \mapsto \sigma(A)|_p$ is a linear isomorphism.

A sub-bundle $\calE$ of $TP$ is called {\it a principal Ehresmann connection} if $TP= \calE \oplus \calV$ and satisfy $r_g \calE_p= \calE_{p\cdot g}$, where $r_g$ denotes the right multiplication of $g \in G$. Equivalently, we can consider {\it a principal connection form}, which is a $\frakg$-valued one-from satisfying the two conditions
$$r^*_g\omega = \Ad(g^{-1}) \omega,\quad  \text{ and }\quad  \omega(\sigma(A)|_p) = A \text{ for any } A \in \frakg, p \in P.$$
There is a one-to-one correspondence between these two structures, in the sense that $\ker \omega$ is a principal Ehresmann connection when $\omega$ is a principal connection form, and for any principal Ehresmann connection $\calE$, we can define a principal connection form by formula
$$\omega(v) = A \text{ for $v \in T_pP$ if } v - \sigma(A)|_p \in \calE_p.$$

Related to a choice of Ehresmann connection $\calE$ we also have horizontal lifts, since the mapping $\tau_*|_{\calE_p}:\calE_p \to T_{\tau(p)}P$ is a linear isometry. For a vector $v \in T_mM$, we define the horizontal lift $h_p v$ of $v$ at $p \in \tau^{-1}(m)$ as the unique element in $\calE_p$ which is projected to $v$ by $\tau_*$.

Let us consider a particular principal bundle over a manifold $M$. For two vector spaces $V$ and $\widehat V$, let $\GL(V, \widehat V)$ be the space of all linear isomorphisms from $V$ to $\widehat V$. For any $m\in M$, we say that a frame at $m$ is a choice of basis $f_1, \dots, f_n$ for $T_mM$. Equivalently, we can consider a frame as a linear map $f \in \GL(\real^n,T_mM)$. The correspondence between the two point of views, is given by
$$f\! \underbrace{(0, \dots, 0, 1, 0, \dots, 0)}_{\text{$1$ in the $j$th coordinate}} = f_j.$$
We write $\calF_m(M) = \GL(\real^n,T_mM)$ for the space of frames in $T_mM$. There is a natural action of $\GL(n)$ on these spaces by composition on the right. Using this action, we define {\it the frame bundle} $\widetilde \tau:\calF(M) \to M$ as the principal $\GL(n)$-bundle with fiber over $m$ being $\calF_m(M)$.

Similarly, if $M$ is an oriented Riemannian manifold, we can define {\it the oriented orthonormal frame bundle} $\tau:F(M) \to M$ as the principal $\SO(n)$-bundle whose fiber is $F_m(M):=\SO(\real^n, T_mM)$. Here, $\real^n$ is furnished with the standard orientation and the Eucliean metric.

Let $\nabla$ be an affine connection on $M$, seen as an operator on vector fields
$(X , Y) \mapsto~\nabla_X Y.$
Then we can associate a principal Ehresmann connection $\calE_\nabla$ on $\widetilde \tau:\calF(M) \to~M$ to $\nabla$, by defining $\calE_\nabla$ to be the distribution on $\calF(M)$ consisting of tangent vectors of smooth curves $f(t)$ such that the vector fields $f_1(t),\dots, f_n(t)$ are all parallel along $m(t):=\widetilde \tau(f(t))$. If $M$ is Riemannian and oriented, and if $\nabla$ is compatible with a metric, then $\calE_\nabla$ can be defined as a distribution on $F(M)$ instead, since both positive orientation and orthonormality are preserved under parallel transport.

We now go to the concrete case where $\nabla$ is the Levi-Civita connection.
Define {\it the tautological one-from} $\theta = (\theta_1, \dots, \theta_n)$ on $F(M)$, as the $\real^n$-valued one-form
$$\theta_j|_f = \tau^*(\flat f_j), \qquad \tau:F(M) \to M,$$
where $\flat:TM \to T^*M$ is the isomorphism induced by the Riemannian metric.
In other words, if $v \in T_fF(M)$, then $\theta_j(v) = \langle f_j, \tau_* v \rangle.$ Denote the $\so(n)$-valued principal connection form corresponding to the Levi-Civita connection by $\omega$. The formulas for the differentials of $\theta$ and $\omega$ are given by the well-known Cartan equations. We express them in notations, that will be helpful for later purposes.

Let $R$ be the Riemann curvature tensor, defined by
$$R(Y_1, Y_2, Y_3, Y_4) = \left\langle R(Y_1,Y_2) Y_3, Y_4 \right\rangle, \text{where }
R(Y_1, Y_2) = \nabla_{Y_1} \nabla_{Y_2} -\nabla_{Y_2} \nabla_{Y_1} - \nabla_{[Y_1,Y_2]}.$$
Furthermore, use it to define the curvature form $\Omega = (\Omega_{ij})$, as the $\so(n)$-valued two-form
\begin{equation} \label{OmegaDef} \Omega_{ij}(v_1,v_2) = R(\tau_* v_1, \tau_* v_2, f_j, f_i), \qquad v_1, v_2 \in T_f F(M).\end{equation}
Then the following relations hold
\begin{equation} \label{Cartanequations}
\begin{array}{rl} d\theta_j + \sum_{i =1}^n\omega_{ji} \wedge \theta_i = &  0, \\ \\
d\omega_{ij} + \sum_{k =1}^n \omega_{ik} \wedge \omega_{kj}  = & \Omega_{ij}, \end{array} \end{equation}
where $\omega_{ij}$ are the matrix entries of $\omega$.
These equations are going to be important in order to understand the connections between geometry and the control system of rolling manifolds.
\begin{remark}
The Cartan equations can also be defined on the frame bundle corresponding to a general affine connection. See, e.g., \cite{Jost, Sharpe} for details.
\end{remark}

From now on, we will adopt the convention that whenever we mention $\real^n$, it will always come furnished with the standard orientation and the Euclidean metric.

\subsection{The rolling distribution and the corresponding sub-Riemannian structure} \label{rollingdistribution}
The tangent vectors of all possible rollings form an $n$-dimensional distribution $D$ on $Q$. We will call this distribution $D$ the rolling distribution. A curve $q(t)$ is a rolling, if and only if, it is horizontal with respect to $D$, i.e. it is absolutely continuous and $\dot{q}(t) \in D_{q(t)}$ for almost any $t$.

We present the following local description of $D$ (see \cite{GGMS} for more details). Let us write the projection of $Q$ to $M \times \widehat M$, as $\overline \pi: Q \to M \times \widehat M.$
Given any sufficiently small neighborhood $U$ on $M$, let $e$ be a local section of the oriented orthogonal frame bundle $F(M)$ with domain $U$. We write this local section as $(e,U)$. Let $(\widehat e, \widehat U)$ be a similar local section of $F(\widehat M)$. We can use these local sections to trivialize the fiber bundle $Q$ over $U \times \widehat U$, with the map
$$q \in \SO(T_mM, T_{\widehat m} \widehat M) \mapsto (m, \widehat m, (q_{ij})) \in U \times \widehat U \times \SO(n), \qquad q_{ij} := \langle \widehat e_j, q e_j \rangle.$$
Then the distribution $D|_{\overline \pi^{-1}(U \times \widehat U)}$, in the above coordinates, 
is spanned by the vector fields
\begin{equation} \label{basisD} \overline e_j: = e_j + qe_j + \sum_{1 \leq \alpha \leq \beta \leq n} \left(\left\langle e_\alpha, \nabla_{e_j} e_\beta \right\rangle - \left\langle qe_\alpha, \nabla_{qe_j} qe_\beta \right\rangle \right) W_{\alpha\beta}^\ell.\end{equation}
where $j =1,\dots, n$. Here, $e_j$ is seen as a vector field on $Q|_{U \times \widehat U} = \overline \pi^{-1}(U \times \widehat U)$ and $qe_j$ stands for the vector field $q \mapsto q (e_j|_{\pi(q)}).$
The vector fields $W_{\alpha\beta}^\ell$ are defined by
\begin{equation} \label{Wbasisl} W_{\alpha\beta}^\ell = \sum_{s=1}^n \left(q_{s\alpha} \frac{\partial}{\partial q_{s\beta}}
- q_{s\beta} \frac{\partial}{\partial q_{s\alpha}} \right).\end{equation}
The symbol $\ell$ here is not a parameter; it simply stands for ``left'' (an explanation of this will follow in Remark \ref{localrightremark}).

Let us consider the optimal control problem finding of a rolling $q(t)$ connecting two configurations $q_0$ and $q_1$, such that the curve $m(t) = \pi(q(t))$ has minimal length. We do this by introducing a metric $\langle \cdot , \cdot \rangle$ on $D$, defined by
\begin{equation}\label{metricD} \langle \overline v_1, \overline v_2 \rangle = \langle \pi_* \overline v_1, \pi_* \overline v_2 \rangle, \qquad \overline v_1, \overline v_2 \in D_q.\end{equation}
With respect to this metric, the $\overline e_j$ in \eqref{basisD} are a local orthonormal basis. Also, from the definition of $D$, we have $\langle \overline v_1, \overline v_2 \rangle = \langle \widehat \pi_* \overline v_1, \widehat \pi_* \overline v_2 \rangle$. Hence, minimizing the length of $m(t)$ is equivalent to minimizing the length of $\widehat m(t) = \widehat \pi (q(t))$ (this is can also be seen from the no-slip condition).
\begin{definition}
A triple $(Q, D, \langle \cdot , \cdot \rangle)$, where $Q$ is a connected manifold, $D$ is a distribution on $Q$ and $\langle \cdot , \cdot \rangle$ is a metric on $D$, is called a sub-Riemannian manifold.
\end{definition}
The sub-Riemannian distance function $d(q_0,q_1)$ between two points is defined as the infimum of the length of all curves which are horizontal to $D$. We can view rollings between configuration from $q_0$ to $q_1$ along a curve of minimal length as a sub-Riemannian length minimizer in $Q$.

For the distance function to be finite, we need that every pair of configurations $q_0$ and $q_1$ can be connected by a curve horizontal to $D$. In other words, we need to determine when rolling without twisting or slipping is a controllable system. Our goal is to give sufficient conditions for this to hold, in terms of geometric invariants on $M$ and $\widehat M$. We will leave the question of finding optimal curves for later research.

\begin{remark} \label{localrightremark}
The vector fields $W_{\alpha\beta}^\ell$ in \eqref{Wbasisl} can be considered as a ``locally left invariant" basis of $\ker \overline \pi_*$. It will be practical to also introduce a ``locally right invariant" analogue. Relative to two chosen local sections $(e, U)$ and $(\widehat e, \widehat U)$, define
\begin{equation} \label{Wbasisr} W_{\alpha\beta}^r = \sum_{s=1}^n \left(q_{\beta s} \frac{\partial}{\partial q_{\alpha s}}
- q_{\alpha s} \frac{\partial}{\partial q_{\beta s}} \right), \quad q_{ij} := \langle \widehat e_i, qe_j \rangle.\end{equation}
Notice that $W_{\alpha\beta}^r =\sum_{l,s}^n q_{\alpha l} q_{\beta s} W_{ls}^\ell.$
\end{remark}

\subsection{Controllability and brackets}
Given an initial configuration $q_0 \in Q$, write $\calO_{q_0}$ for all points in $Q$ that are reachable by a rolling starting from $q_0$. This will be the orbit of $D$ at $q_0$, which coincides with the reachable set of $D$, since $D$ is a distribution (see, e.g., \cite{AS, Jur} for details). The Orbit Theorem \cite{Hermann,Suss} tells us that 
$\calO_{q_0}$ is a connected, immersed submanifold of $Q$, but also that the size can be approximated by the brackets of $D$. Define the $C^\infty(Q)$-module $\Lie \, D$ as the limit of the process
$$D^1 = \Gamma(D), \qquad D^{k+1} = D^k + [D, D^k].$$
$\Gamma(D)$ denotes the sections of $D$. Let $D_q^k$ and $\Lie_q D$ be the subspaces of $T_qQ$ obtained by evaluating respectively $D^k$ and $\Lie D$ at $q$.
Then, for any $q \in \calO_{q_0}$,
\begin{equation} \label{lowerLie} \Lie_q D \subseteq T_q\calO_{q_0}.\end{equation}
In particular, it follows from \eqref{lowerLie}, that if $D$ is {\it bracket generating} at $q$,
i.e., if $\Lie_q D = T_qQ$, then $\calO_{q_0}$ is an open submanifold of $Q$, and we say that we have {\it local controllability} at $q_0$.
If $\calO_{q_0}=Q$ for one (and hence all) $q_0 \in Q$, the system is called {\it completely controllable}.

The least amount of control happens when $\calO_{q_0}$ is $n$-dimensional submanifold. As a consequence of the Orbit theorem and Frobenius theorem, this happens if and only if $D|_{\calO_{q_0}}$ is involutive, that is, if
$\Lie_q D = D_q, \text{ for every } q \in \calO_{q_0}.$

The focus of this paper will be to provide results of controllability, by investigating when the distribution $D$ will be bracket generating at a given point $q$.

\begin{remark} \label{locfg}
When $D$ is not bracket generating at $q$, $\Lie_q D$ will in general only give us a lower bound for the size of $\calO_q$. However, if $\Lie D$ is locally finitely generated as a $C^\infty(Q)$-module, i.e., has a finite basis of vector field when restricted to a sufficiently small neighborhood, then the equality holds in \eqref{lowerLie}. \end{remark}

\begin{remark}
We will use the notation introduced here for distributions in general, not just for the rolling distribution.
\end{remark}

\subsection{Relationship between frame bundles and rolling} \label{frametorolling}
Consider the linear Lie algebra $\so(n)$. For integers $\alpha$ and $\beta$ between 1 and $n$ and not equal, let $w_{\alpha\beta}$ be the matrix with $1$ at entry $\alpha\beta$, -1 at entry $\beta\alpha$, and zero at all other entries. Clearly $w_{\alpha\beta} = - w_{\beta\alpha}$. Define $w_{\alpha\alpha} = 0.$
The commutator bracket between these matrices are given by
\begin{equation} \label{bracketson}
[w_{\alpha\beta}, w_{\kappa\lambda}] = \delta_{\beta,\kappa} w_{\alpha\lambda}
+ \delta_{\alpha,\lambda} w_{\beta\kappa} - \delta_{\alpha,\kappa} w_{\beta\lambda}
- \delta_{\beta,\lambda} w_{\alpha\kappa},\end{equation}
The collection of all $w_{\alpha\beta}$ with $\alpha < \beta$ form a basis for $\so(n)$.

Consider the principal bundle $F(M)\to M$. We introduce the following notation.  Write $\sigma_{\alpha\beta}$ for the vector fields on $F(M)$ corresponding to $w_{\alpha\beta}$ in the sense of \eqref{VAcor}. Let $\calE$ be the principal Ehresmann connection corresponding to the affine Levi-Civita connection on $M$, with corresponding principal connection form $\omega$. As before, we denote the tautological one-form by $\theta$. By using horizontal lifts with respect to $\calE$, we can define vector fields $X_j$ on $F(M)$ by
\begin{equation} \label{basisEhresmann} X_j|_f := h_f f_j,\end{equation}
which satisfy $\theta_i(X_j) = \delta_{i,j}.$ Hence, the tangent bundle of $F(M)$ is trivial, since it is spanned by $\{X_j\}_{j=1}^n$ and $\{\sigma_{\alpha\beta}\}_{\alpha<\beta}$.  Finally, we define $\widehat \sigma_{\alpha\beta}, \widehat X_j, \widehat \omega$ and $\widehat \theta$ similarly on $F(\widehat M)$.

The configuration space $Q$ may be identified with $F(M) \times F(\widehat M)$ quotiented out by the diagonal action of $\SO(n)$. Let $\varpi$ denote the principal $\SO(n)$-bundle
$$\varpi: F(M) \times F(\widehat{M}) \to F(M) \times F(\widehat{M})/\SO(n) \cong Q.$$
Then $\varpi(f,\widehat f) = q$, if $\widehat f = q \circ f.$
By viewing $\omega, \widehat \omega, \theta$ and $\widehat \theta$ as forms on $F(M)\times F(\widehat M)$, we are able to obtain the following result.

\begin{theorem} \label{mathcalDtoD}
Let
$$\calD = \ker \omega \cap \ker \widehat \omega \cap \ker(\theta - \widehat \theta),$$
and let $D$ be the rolling distribution. Then $\varpi_* \calD = D$, and the map is a linear isomorphism on every fiber.
\end{theorem}
\begin{proof}
From its definition, it is clear that $\{X_j + \widehat X_j\}_{j=1}^n$ is a basis for $\calD$.

Choose any pair of local section $(e, U)$ and $(\widehat e, \widehat U)$ of respectively $F(M)$ and $F(\widehat M)$. Give $F(M) \times F(\widehat M)|_{U \times \widehat U}$ local coordinates by associating the pair of frames
$(f, \widehat f)$ to the element
\begin{equation} \label{localtriv} \Big(m, \widehat m, \big(f_{ij}\big), \big(\widehat f_{ij}\big) \Big) \in U \times \widehat U \times \SO(n) \times \SO(n),\end{equation}
if $f_j = \sum_{i=1}^n f_{ij} e_i|_m$ and $\widehat f_{ij} = \sum_{i=1}^n \widehat f_{ij} \widehat e_i|_{\widehat m}$
holds.

Relative to this trivialization, we can define left and right vector field on each of the $\SO(n)$-factors.
On the first, define
\begin{equation} \label{leftrightvector} \Psi_{\alpha\beta}^\ell  = \sum_{s=1}^n \left(f_{s \alpha} \frac{\partial}{\partial f_{s\beta}} - f_{s\beta} \frac{\partial}{\partial f_{s\alpha}}\right), \quad
\Psi_{\alpha\beta}^r  = \sum_{s=1}^n \left(f_{\beta s} \frac{\partial}{\partial f_{\alpha s}} - f_{\alpha s} \frac{\partial}{\partial f_{\beta s}}\right).  \end{equation}
Notice that $\Psi_{\alpha\beta}^r =\sum_{l,s=1}^n f_{\alpha l} f_{\beta s} \Psi_{ls}^\ell.$
Remark also that $\Psi_{\alpha\beta}^\ell$ is just the restriction of $\sigma_{\alpha\beta}$ to $F(M) \times F(\widehat M)|_{U \times \widehat U}$, while $\Psi_{\alpha\beta}^r$ depends on the chosen local section $e$.
Define $\widehat \Psi_{\alpha\beta}^\ell$ and $\widehat \Psi_{\alpha\beta}^r$ analogously on the second $\SO(n)$-factor.

Restricted to $F(M) \times F(\widehat M)|_{U \times \widehat U}$ and in the local coordinates \eqref{localtriv}, the vector fields $X_j$ and $\widehat X_j$ can be written as
\begin{equation}\label{eq:localX}X_j = \sum_{s = 1}^n f_{sj} \Big(e_s - \sum_{\alpha < \beta} \Gamma_{s\beta}^\alpha \Psi_{\alpha\beta}^r\Big),
\qquad
\widehat X_j = \sum_{s=1}^n \widehat f_{sj} \Big(\widehat e_s - \sum_{\alpha < \beta} \widehat \Gamma_{s\beta}^\alpha \widehat \Psi_{\alpha\beta}^r\Big),\end{equation}
where $\Gamma_{i\beta}^\alpha = \left\langle e_\alpha, \nabla_{e_i} e_\beta \right\rangle$ and $\widehat \Gamma_{i\beta}^\alpha = \left\langle \widehat e_\alpha, \nabla_{\widehat e_i} \widehat e_\beta \right\rangle$.

We now turn to the image of $T(F(M) \times F(\widehat M)|_{U \times \widehat U})$ under $\varpi_*$.
Define $q_{ij}, \overline e_j, W_{\alpha\beta}^\ell$ and $W_{\alpha\beta}^r$ on $Q|_{U \times \widehat U}$ as in Section \ref{rollingdistribution}. Remark \ref{localrightremark} allows us to rewrite $\overline e_j$ on the form
$$ \overline e_j = e_j + q e_j + \sum_{\alpha < \beta} \Big(\Gamma_{j\beta}^\alpha W_{\alpha\beta}^\ell - \sum_{s=1}^n q_{sj} \widehat \Gamma_{s\beta}^\alpha W_{\alpha\beta}^r \Big)$$

Locally the mapping $\varpi$ can be described as
$$\varpi: \Big(m, \widehat m, (f_{ij}), (\widehat f_{ij}) \Big) \mapsto \left(m, \widehat m, \left(q_{ij} \right) \right),\quad q_{ij} = \sum_{s=1}^n \widehat f_{is} f_{js}.$$
and the action on the tangent vectors is given by formulas
$$\begin{array}{rrcl} \varpi_*: & \begin{array}{rcl} e_i & \mapsto & e_i \\ \widehat e_i & \mapsto & \widehat e_i \\ \Psi_{\alpha\beta}^r & \mapsto & - W_{\alpha\beta}^\ell \\ \widehat \Psi_{\alpha\beta}^r & \mapsto&  W_{\alpha\beta}^r \end{array} \end{array}. $$
From this it is clear that $\varpi_* \calD = D$ and that the map is injective of each fiber, since
$$\overline e_j = \varpi_* \sum_{s=1}^n f_{js}\left( X_s + \widehat X_s\right).$$
\end{proof}

From the form of the distribution $\calD$, we obtain the following interpretation of rolling without twisting or slipping.
\begin{corollary} \label{stochasticroll}
Let $q(t)$ be a rolling without twisting or slipping. Let $(f(t), \widehat f(t))$ be any lifting of $q(t)$ to a curve in $F(M) \times F(\widehat M)$ that is horizontal to $\calD$, and define $m(t)$ and $\widehat m(t)$ as the respective projections to $M$ and $\widehat M$. Then $(f(t), \widehat f(t))$ satisfy the following
\begin{itemize}
\item[(i)] (No twist condition) Every vector field $f_j(t)$ is parallel along $m(t)$.
Every vector field $\widehat f_j(t)$ is parallel along $\widehat m(t)$.
\item[(ii)] (No slip condition) For almost every $t$,
$$f^{-1}(t)\big(\dot m(t) \big) = \widehat f^{-1}(t)\big(\dot{\widehat m}(t)\big).$$
\end{itemize}

Furthermore, if $(f(t), \widehat f(t))$ is any absolutely continuous curve in $F(M) \times F(\widehat M)$, satisfying (i) and (ii), then $\varpi (f(t), \widehat f(t) )$ is a rolling without twisting or slipping.
\end{corollary}
\begin{proof} (i) follows from the definition of the principal Ehresmann connections $\ker \omega$ and $\ker \widehat \omega$. (ii) is exactly the requirement for a curve to be in $\ker(\theta - \widehat \theta)$.
\end{proof}

The main advantage of the viewpoint given in Theorem \ref{mathcalDtoD}, is that it will help us to compute $\Lie_q D$.
\begin{corollary}
$\varpi_* \calD^k_q = D^k_q$ for any $q \in Q, k \in \mathbb{N}.$
\end{corollary}
\begin{proof}
We only need to show this locally. Introduce local coordinates as in the proof of Theorem \ref{mathcalDtoD} and let $\overline e_j$ be as in \eqref{basisD}. Then
\begin{align*}
\left[\overline e_i, \overline e_j\right] & = \varpi_* \left[\sum_{s = 1}^n f_{is} \left(X_s + \widehat X_s \right), \sum_{s = 1}^n f_{js} \left(X_s + \widehat X_s \right)\right],
\end{align*}
and since $\sum_{s = 1}^n f_{is} \left(X_s + \widehat X_s \right)$ is a local basis for $\calD$, it follows that $D^2 = \varpi_* \calD^2$ locally. The rest follows by induction.
\end{proof}
Since $[X_i, \widehat X_j] = 0$, computations of brackets of $\calD$, and hence of brackets $D$, can be reduced to mostly computing brackets of sections in the Ehresmann connections corresponding to the Levi-Civita connections of the manifolds involved.

\subsection{Remark on previous descriptions of rolling using frame bundles}
The description of rolling given in Theorem \ref{mathcalDtoD}, looks very similar to the definition of rolling without twisting or slipping found in \cite{BH} for dimension 2. Here, the description of a rolling is in terms of the distribution $\widetilde \calD := \calD \oplus \ker \varpi_*$, which can also be described as
\begin{equation} \label{BryantHsu} \widetilde \calD = \ker (\omega - \widehat \omega) \cap \ker(\theta - \widehat \theta).\end{equation}
In \cite{BH}, a rolling of a pair of 2-dimensional manifold is defined as a curve in $Q$, that is horizontal to $\varpi_* \widetilde \calD$, where $\widetilde \calD$ is defined in terms of \eqref{BryantHsu}.

We could have used $\widetilde \calD$ for our computation, since clearly $\varpi_* \widetilde \calD^k_q = D^k_q$ also,
and $D$ is bracket generating at a point $q \in Q$ if and only if $\widetilde \calD$ is bracket generating at any (and hence every) $(f,\widehat f) \in \varpi^{-1}(q)$. However, since $[\calD^k, \ker \varpi_*] \subset \calD^k + \ker \varpi_*,$ the additional brackets are not of any interest.

The definition of rolling or ``rolling without slipping" in probability theory is defined on frame bundles \cite{Hsu}, and is equivalent to considering curves in $F(M) \times F(\widehat M)$ that are horizontal to $\calD$ for the special case when $\widehat M$ is $\real^n$ with the Euclidean metric and standard orientation.

\section{Brackets of $\calD$}
\subsection{Tensors on $M$ and associated vector fields}
We introduce a general type of vector fields associated to tensors. By giving a general formula for the brackets of these, we essentially determine all the brackets of $\calD$. Let $T$ be a tensor on $M$. We will only consider tensors defined on vectors. To any tensor $k$-tensor on $M$, we can associate the functions
$$\begin{array}{rccc} \scrE_{i_1, \dots i_k}(T): & F(M) & \mapsto & \real \\
& f & \mapsto & T(f_{i_1}, \dots, f_{i_k}) \end{array}, \qquad 1 \leq i_s \leq n.$$
If $k = 2 +l$, and $T$ is antisymmetric in the first two arguments, we can define vector fields on $F(M)$ by
$$\scrW_{i_1,\dots, i_l}(T) = \sum_{1 \leq \alpha < \beta \leq n} \scrE_{\alpha,\beta,i_1, \dots, i_l}(T) \sigma_{\alpha\beta},$$

If $\widehat T$ is a tensor on $\widehat M$, we define $\scrE_{i_1, \dots i_k}(\widehat T)$ and $\scrW_{i_1,\dots, i_l}(\widehat T)$ similarly as respectively functions and vector fields on $F(\widehat M)$.

\begin{lemma} \label{WErelations}
\begin{itemize} Let $X_k$ be defined as in \eqref{basisEhresmann}.
\item[(a)] For any $l$-tensor $T$,
$$X_k \scrE_{i_1,\dots, i_l}(T) = \scrE_{i_1,\dots, i_l,k}(\nabla T).$$
\item[(b)] For any $2+l$-tensor $T$, that is antisymmetric in the first two arguments
$$\left[ X_k, \scrW_{i_1,\dots, i_l}(T) \right] = \scrW_{i_1,\dots, i_l,k}(\nabla T)
- \sum_{s=1}^n \scrE_{s, k, i_1, \dots, i_l}(T) X_s.$$
\end{itemize}
\end{lemma}

\begin{proof}
\begin{itemize}
\item[(a)] We use a local representation of $X_j$. Consider the formula given in \eqref{eq:localX}, and write $\sum_{s=1}^n f_{sj} e_s$ as just $f_j$. Since $\Psi^\ell_{\alpha\beta}$ is just the representation of $\sigma_{\alpha\beta}$ in local coordinates, we can write $X_j$ as $$X_j = f_j - \sum_{1 \leq \alpha \leq n} \langle f_{\alpha}, \nabla_{f_j} f_\beta \rangle \sigma_{\alpha\beta}.$$ In this notation, remark first that $\sigma_{\alpha\beta} f_i = \delta_{\beta,i} f_{\alpha} - \delta_{\alpha,i} f_{\beta},$
which gives us
\begin{align*} \sum_{1 \leq \alpha < \beta \leq n} \left\langle f_\alpha, \nabla_{f_k} f_\beta \right\rangle \sigma_{\alpha\beta} f_{i_j} = & \frac12 \sum_{\alpha, \beta = 1}^n \left\langle f_\alpha, \nabla_{f_k} f_\beta \right\rangle (\delta_{\beta,i_j} f_{\alpha} - \delta_{\alpha,i_j} f_{\beta}) \\
= & \sum_{\alpha = 1}^n \left\langle f_\alpha, \nabla_{f_k} f_{i_j} \right\rangle f_\alpha 
=  \nabla_{f_k} f_{i_j}.
\end{align*}
Then the result follows from realizing that
\begin{align*} X_k \scrE_{i_1,\dots, i_l}(T) & =  f_k T(f_{i_1}, \dots, f_{i_l}) - \sum_{1 \leq \alpha < \beta \leq n}  \left\langle f_\alpha, \nabla_{f_k} f_\beta \right\rangle \sigma_{\alpha\beta} T(f_{i_1}, \dots, f_{i_l}) \\
& =  f_k T(f_{i_1}, \dots, f_{i_l}) - T(\nabla_{f_k} f_{i_1}, \dots, f_{i_l}) - T(f_{i_1}, \nabla_{f_k} f_{i_2}, \dots, f_{i_l}) \\ & \quad - \cdots - T(f_{i_1}, f_{i_2}, \dots, \nabla_{f_k} f_{i_l}) \\
& =  \nabla T(f_{i_l}, \dots, f_{i_l}, f_k).\end{align*}

\item[(b)]
The brackets $[\sigma_{\alpha\beta} , \sigma_{\kappa\lambda}]$ are, by definition, given by the same relations as described in \eqref{bracketson}. We will continue by the following computations.
\begin{align*} & [X_k, \scrW_{i_1,\dots, i_l}(T)]
\\ = & \, \frac{1}{2} \sum_{\kappa, \lambda =1}^n X_k\left( \scrE_{\kappa, \lambda, i_1, \dots, i_l}(T) \right) \sigma_{\kappa\lambda} - \frac{1}{4} \sum_{\alpha, \beta, \kappa, \lambda =1}^n \left\langle f_\alpha, \nabla_{f_k} f_\beta \right\rangle \scrE_{\kappa,\lambda,i_1,\dots,i_l}(T) \left[\sigma_{\alpha\beta} , \sigma_{\kappa\lambda} \right] \\
& - \frac{1}{2} \sum_{\kappa, \lambda =1}^n \scrE_{\kappa,\lambda,i_1,\dots,i_l}(T) \sigma_{\kappa\lambda} f_k 
+ \frac{1}{4} \sum_{\alpha, \beta, \kappa, \lambda =1}^n  \scrE_{\kappa,\lambda,i_1,\dots,i_l}(T) \sigma_{\kappa\lambda} \left(\left\langle f_\alpha, \nabla_{f_k} f_\beta \right\rangle \right) \sigma_{\alpha\beta} \end{align*}
\begin{align*}= & \,  \frac{1}{2} \sum_{\kappa, \lambda =1}^n \scrE_{\kappa, \lambda, i_1, \dots, i_l}(\nabla T) \sigma_{\kappa\lambda}\\
& + \frac{1}{2} \sum_{\kappa, \lambda =1}^n \left(T(\nabla_{f_k} f_\kappa, f_\lambda, f_{i_1}, \dots, f_{i_l}) + T(f_\kappa, \nabla_{f_k} f_\lambda, f_{i_1}, \dots, f_{i_l}) \right)\sigma_{\kappa\lambda} \\
& - \sum_{s =1}^n \scrE_{s,k,i_1,\dots,i_l}(T) f_s + \frac{1}{2} \sum_{\alpha, \beta, s =1}^n  \scrE_{s,\alpha,i_1,\dots,i_l}(T) \left(\left\langle f_s, \nabla_{f_k} f_\beta \right\rangle \right) \sigma_{\alpha\beta}\\
& + \frac{1}{2} \sum_{\alpha, \beta, s =1}^n  \scrE_{s,k,i_1,\dots,i_l}(T) \left(\left\langle f_\alpha, \nabla_{f_s} f_\beta \right\rangle \right) \sigma_{\alpha\beta}\\
& + \frac{1}{2} \sum_{\alpha, \beta, s =1}^n  \scrE_{s,\beta,i_1,\dots,i_l}(T) \sigma_{w_{\kappa\lambda}} \left(\left\langle f_\alpha, \nabla_{f_k} f_s \right\rangle \right) \sigma_{\alpha\beta}\\
= & \, \scrW_{i_1,\dots, i_l,k}(\nabla T)
- \sum_{s=1}^n \scrE_{s, k, i_1, \dots, i_l}(T) X_s
\end{align*}
\end{itemize} 
\end{proof}

The next lemma gives an explanation for the introduction of the above notation.
\begin{lemma} \label{Xbracketlemma}
$[X_i,X_j] = -\scrW_{ij}(R) \text{ for } i,j = 1 \dots n.$
\end{lemma}
\begin{proof}
This lemma is an easy consequence of the Cartan equations. Since $\ker \theta \cap \ker \omega$ only contains the zero section of $TF(M)$, we can show equality in the above equation by evaluating the left and right hand side by $\theta$ and $\omega$ and see that it produces the same result. Evaluating the left hand side, we get
\begin{align*}\theta([X_i,X_j]) &= - d\theta(X_i, X_j) = 0,\end{align*}
\begin{align*}\omega([X_i,X_j]) &= - d\omega(X_i, X_j) = \sum _{1 \leq \alpha < \beta \leq n} \Omega(X_i, X_j)
\\& = -\sum _{1 \leq \alpha < \beta \leq n} R(f_\alpha, f_\beta, f_i, f_j) w_{\alpha,\beta}. \end{align*}
which is obviously what we get from evaluating the right hand side.
\end{proof}

\begin{remark} \label{rmkEhresmann}
Combining these two lemmas, we get a way to express the commutators of the Ehresmann connection. Let $\calE = \ker \omega$ be the Ehresmann connection corresponding to the Levi-Civita connection. Then
$$\calE^{k+2} = \calE^{k+1} + \spn \left\{\scrW_{i_1,\dots,i_{k+2}}(\nabla^{k}R) \right\}_{i_1, \dots, i_{k+2} =1}^n,$$
for $k \geq 0.$ As a consequence of this, we get the well know Ambrose-Singer theorem, see \cite{AmSi53},\cite[App C]{Mon02}, that the subalgebra spanned by elements $\omega|_f\left(\scrW_{i_1,\dots,i_{k+2}}(\nabla^{k}R)\right)$ is contained in the homology algebra at $f$.
\end{remark}

We adopt the convention that if the elements in the collection are vector fields, ``span'' means the span over smooth functions (so in Remark \ref{rmkEhresmann}, it means over $C^\infty(F(M))$). If the elements are vectors, the span is over the real numbers.

\subsection{Obtaining the brackets for $\calD$}
Computing the brackets on $\calD$ is a bit more complicated than each individual Ehresmann connection, since it is harder to know whether or not two vectors fields are equal $\mathrm{mod } \spn \{X_j + \widehat X_j\}_{j=1}^n$ rather than just $\mathrm{mod } \spn \{X_j \}_{j=1}^n$.
We illustrate this by computing the two next brackets.

\begin{lemma} \label{XXXbracket}
\begin{itemize}
\item[(a)] $[X_k, [X_i, X_j]] = - \scrW_{ijk}(\nabla R) + \sum_{s=1}^n \scrE_{skij}(R) X_s.$
\item[(b)] Let $R^2$ be the 6-tensor on $M$ defined by
$$R^2(Y_\alpha, Y_\beta, Y_{i_1}, Y_{i_2}, Y_{i_3}, Y_{i_4}) = R(R(Y_\alpha, Y_\beta) Y_{i_1}, Y_{i_2},  Y_{i_3}, Y_{i_4}).$$
Then
\begin{align*} [X_l, [X_k, [X_i,X_j]] =& -\scrW_{ijkl}(\nabla^2 R) + \scrW_{lkij}(R^2)
\\ &+ \sum_{s=1}^n \left(\scrE_{ijslk}(\nabla R) - \scrE_{ijskl}(\nabla R)\right) X_s. \end{align*}
\end{itemize}
\end{lemma}
The reason for the notation $R^2$ will be clearer in Section \ref{seclocsym}.

\begin{proof}
Statement (a) follows directly from Lemma \ref{WErelations}. By Lemma \ref{WErelations} we also have
\begin{align*} & [X_l, [X_k, [X_i,X_j]]] \\
=& - \scrW_{ijkl}(\nabla^2 R) + \sum_{s=1}^n \scrE_{slijk}(\nabla R) X_s  + \sum_{s=1}^n \scrE_{skijl}(\nabla R) X_s  + \sum_{s=1}^n \scrE_{skij}(R) \scrW_{ls}(R) \\
=& -\scrW_{ijkl}(\nabla^2 R) + \scrW_{lkij}(R^2)+ \sum_{s=1}^n \left(\scrE_{ijslk}(\nabla R) - \scrE_{ijskl}(\nabla R)\right) X_s.
\end{align*}
\end{proof}
We can continue this procedure, computing more of the brackets using Lemma~\ref{WErelations}. However, these will become more and more complicated. Also, for a general pair of manifolds, it is hard to determine which brackets actually give us something new, that is, something that could not be expressed as linear combinations of previously obtained vectors. Rather than giving the total picture, we will therefore focus on giving some sufficient conditions, which are usually more simple to check.

\section{Sufficient conditions for controllability}

Let $R$ and $\widehat R$ be the curvature tensor on $M$ and $\widehat M$ respectively.
Define a new 4-tensor of elements in $D$, by
$$\overline{R} = \pi^*(R) - \widehat \pi^*(\widehat R).$$
Remark that $\overline R$ may also be seen as a bilinear map of two elements in $\bigwedge^2 D$.
Use $\overline{\nabla R}$ to denote the $5$-tensor on $D$, defined by $\pi^*(\nabla R) - \widehat \pi^*(\nabla \widehat R).$
Finally, introduce a bundle morphism $\overline{\calR}: \bigwedge^2 D \to \bigwedge^2 D^*,$
so that
\begin{equation} \label{overlineOmega} \overline{\calR}(\overline{\xi}_1)(\overline{\xi}_2) = \overline{R}\Big(\overline \xi_2 \, , \overline \xi_2\Big)
\qquad \overline \xi_1, \overline \xi_2 \in \bigwedge^2 D_q.\end{equation}

\subsection{Projection of the results on $\calD$}
From the discussion in previous section, we have the following formulations for the brackets of $D$.
\begin{lemma} \label{backtoD}
Let $(e,U)$ and $(\widehat e, \widehat U)$ be two local sections of $F(M)$ and $F(\widehat M)$, respectively.
Then on $Q|_{U \times \widehat U}$, in terms of the notation introduced in Section \ref{rollingdistribution},
$$D^2 = D^1 \oplus \spn \bigg\{ \sum_{1\leq \alpha < \beta \leq n} \overline{R}(\overline{e}_\alpha, \overline{e}_\beta, \overline e_i, \overline e_j)W_{\alpha\beta}^\ell\bigg\}_{i,j=1}^n$$
\begin{align*} D^3 = D^2 + \spn \bigg\{ & \sum_{1\leq \alpha < \beta \leq n} \overline{\nabla R}(\overline{e}_\alpha, \overline{e}_\beta, \overline e_i, \overline e_j, \overline e_k) W_{\alpha\beta}^\ell + qR(e_i, e_j) e_k - \widehat{R}(qe_i, qe_j) qe_k \\
& + \sum_{1\leq \alpha < \beta \leq n} \left\langle qe_\alpha, \nabla_{qR(e_i, e_j) e_k - \widehat{R}(qe_i, qe_j) qe_k} q e_\beta\right\rangle \bigg\}_{i,j,k=1}^n.
\end{align*}
\end{lemma}

\begin{proof}
The formula for $D^2$ follows directly from Lemma \ref{Xbracketlemma} and the local formulation of $\varpi_*$ given in the proof of Theorem \ref{mathcalDtoD}. To see how the expression for $D^3$ follows from Lemma \ref{XXXbracket} (a), observe first that
\begin{align*} &\sum_{s=1}^n \left(\scrE_{skij}(R) X_s + \scrE_{skij}(\widehat R) \widehat X_s \right)
=  - \sum_{s=1}^n \left(\scrE_{ijks}(R) X_s + \scrE_{ijks}(\widehat R) \widehat X_s \right) \\
= & \sum_{s=1}^n \left(\scrE_{ijks}(R)  - \scrE_{ijks}(\widehat R)\right) \widehat X_s \quad (\text{mod } \calD).\end{align*}
Furthermore
\begin{align*}
& \varpi_*  \sum_{s, \mu, \lambda, \kappa, =1}^n f_{i\mu} f_{j\lambda} f_{k\kappa} \left(\scrE_{\mu\lambda\kappa s}(R)  - \scrE_{\mu\lambda\kappa s}(\widehat R)\right) \widehat X_s \\
= \, &  q R(e_i, e_j) e_k - \widehat{R}(qe_i, qe_j) qe_k  - \sum_{1 \leq \alpha < \beta \leq n}^n \left\langle q e_\alpha, \nabla_{q R(e_i, e_j) e_k} qe_\beta \right\rangle W_{\alpha\beta}^\ell \\
& - \sum_{1 \leq \alpha < \beta \leq n}^n \left\langle q e_\alpha, \nabla_{R(q e_i, q e_j) q e_k} qe_\beta \right\rangle W_{\alpha\beta}^\ell.
\end{align*}
\end{proof}

\begin{corollary} \label{corD3dim}
Define a bundle morphism $\overline{\Xi}: D \oplus \bigwedge^2 D \to D^*$ by
$$\overline{\Xi}(\overline v, \overline \xi) = \iota_{\overline v}\circ \overline{\calR}(\overline \xi),
\qquad \overline{v} \in D_q, \overline{\xi}\in \bigwedge^2 D_q,$$
where $\iota_{\overline v}: \overline\eta \mapsto \overline \eta(\overline v, \cdot )$
for any $\overline v \in D_q$ and two form $\overline \eta$. Then $\dim D^3_q \geq n+ \rank \overline{\Omega}|_q +\rank \overline{\Xi}|_q.$
\end{corollary}
\begin{proof}
Given point $q \in Q$, introduce local coordinates in a neighborhood of $q$, in the way demonstrated in Section \ref{rollingdistribution}. We will also keep the same notation from the previous mentioned section.
Then, all we need to show is that for a given $q \in Q$, the dimension of
\begin{equation} \label{spanspace}
\spn \left\{q R(e_i, e_j) e_k(\pi(q)) - R(qe_i, qe_j) qe_k(\widehat{\pi}(q)) \right\}_{i,j,k =1}^n \subset
T_{\widehat{\pi}(q)} \widehat{M},\end{equation}
is equal to $\rank \overline \Xi_q$.

Introduce isomorphisms $\flat:D \to D^*$ and $\sharp:D^* \to D$, relative to the metric defined in \eqref{metricD}. Use the same symbols for the isomorphisms between the tangent bundle and the cotangent bundle on $M$ and $\widehat M$, induced by their respective metrics. Observe that
\begin{align*} \overline{\Xi}(\overline e_k, \overline e_i \wedge \overline e_j) 
& = \pi^*\Big(\flat(R(e_i, e_j) e_k)\Big) - \widehat \pi^* \Big(\flat(\widehat{R}(qe_i,qe_j)qe_k\Big) \\
& = \flat \overline Y - \flat \overline Z,
\end{align*}
where
\begin{align*} \overline Y= & \, R(e_i, e_j)e_k + q R(e_i, e_j)e_k \\
&+ \sum_{1 \leq \alpha < \beta \leq n} \left(
\left\langle e_\alpha, \nabla_{R(e_i, e_j) e_k} e_\beta \right\rangle -
\left\langle qe_\alpha, \nabla_{qR(e_i, e_j) e_k} qe_\beta \right\rangle \right) W_{\alpha\beta}^\ell,\end{align*}
\begin{align*} \overline Z= & \, q^{-1} \widehat{R}(q e_i, qe_j) qe_k + \widehat{R}(q e_i, qe_j) qe_k \\
&+ \sum_{1 \leq \alpha < \beta \leq n} \left(
\left\langle e_\alpha, \nabla_{q^{-1} \widehat{R}(q e_i, qe_j) qe_k} e_\beta \right\rangle -
\left\langle qe_\alpha, \nabla_{\widehat{R}(q e_i, qe_j) qe_k} qe_\beta \right\rangle \right)W_{\alpha\beta}^\ell.
\end{align*}
From this, it becomes clear that
$\widehat{\pi}_* \sharp$ is a bijective linear map from the image of $\overline{\Xi}_q$ to \eqref{spanspace}.
\end{proof}

\subsection{Sufficient condition in terms of the curvature tensor and sectional curvature}
As mentioned before, there is a strong connection between controllability and geometry in the two dimensional case.
\begin{theorem}[\cite{AS, BH}] \label{AS2dim}
For $q \in Q$, let $\varkappa_q$ denote the Gaussian curvature of $M$ at $\pi(q)$,
and let $\widehat{\varkappa}_q$ denote the Gaussian curvature of $\widehat{M}$ at $\widehat{\pi}(q)$.
Then
$$\dim \calO_q = 5, \qquad \text{if and only if} \qquad \varkappa - \widehat{\varkappa} \not \equiv 0 \text{ on } \calO_q.$$
If $\varkappa - \widehat{\varkappa} \equiv 0$ on $\calO_q$, then $\dim \calO_q = 2.$
\end{theorem}

The ``if and only if" in the above theorem follows from the fact that in two dimensions, the rolling distribution $D$ at a point $q$, is either bracket generating or involutive. This does not hold in higher dimensions, however, but we are able to present the following generalization.

\begin{definition}
The smallest integer $k$ such that $D^k_q = \Lie_q D$ is called the step of $D$ at $q$. 
\end{definition}

\begin{theorem} \label{maintheorem}
Let $\overline \calR$ be as defined in \eqref{overlineOmega}. Then, for any element $q, q_0 \in Q$, the following holds.
\begin{itemize}
\item[(a)] $\dim \calO_{q_0} = n$ if an only if $\overline{\calR}|_{\calO_{q_0}} \equiv 0$.
\item[(b)] If $\overline \calR_q$ is an isomorphism, then $D$ is bracket generating of step 3 at $q$.

Hence, if $\calO_{q_0}$ contains a point $q$, so that $\overline \calR|_q$ is an isomorphism, then $\calO_{q_0}$ is an open submanifold.
\end{itemize}
\end{theorem}

\begin{remark}
The statement in Theorem \ref{maintheorem} (a) was also presented in \cite[Cor. 5.28]{Chitour}.
By combining \cite[Cor. 5.26]{Chitour} with \cite[Prop. 5.17]{Chitour}, and doing some simple calculations, we can also obtain the result of Theorem 3(b), however, this is not stated.
The proof is presented here, since the approach in \cite{Chitour} differ from ours, and since the results were obtained independently.
\end{remark}

\begin{proof}
Statement (a) becomes obvious from Lemma \ref{backtoD}.
To prove (b), let $\overline \pi(q) = (m, \widehat m)$, where $\overline \pi:Q \to M \times \widehat M$ is the projection.
Pick respective local sections $(e,U)$ and $(\widehat e, \widehat U)$ of $F(M)$ and $F(\widehat M)$ around $m$ and $\widehat m$ and use them to introduce the local coordinates defined in Section \ref{rollingdistribution}. Since $\overline \calR|_q$ is an isomorphism, we know that
$$\spn\bigg\{ \sum_{1\leq \alpha < \beta \leq n} \overline{R}(\overline{e}_\alpha, \overline{e}_\beta, \overline e_i, \overline e_j)W_{\alpha\beta}^\ell \Big|_q \bigg\}_{i,j=1}^n = \spn \big\{W_{ij}^\ell(q)\big\}_{i,j=1}^n,$$
$$\spn \left\{ \left(qR(e_i, e_j) e_k - \widehat{R}(qe_i, qe_j) qe_k\right)\Big|_q \right\}_{i,j,k=1}^n = \spn \big\{qe_i(m)\big\}_{i=1}^n.$$
Lemma \ref{backtoD} then tells us that
\begin{align*} D_q^3 &= \spn\{\overline e_j(q), qe_i(m),  W_{\alpha\beta}^\ell(q) \}_{i,j,\alpha\beta=1}^n \\
& = \spn\{ e_j(m), \widehat e_i(\widehat m),  W_{\alpha\beta}^\ell(q) \}_{i,j,\alpha\beta=1}^n = T_q Q.\end{align*}
\end{proof}

To state that $\overline \calR_q$ is an isomorphism is equivalent to claiming that $\overline R_q$ induces a pseudo-inner product on $\bigwedge^2 D_q$, i.e. a nondegenerate bilinear map.
Therefore, we have the following way we can check if $\overline \calR$ is an isomorphism at $q$.
\begin{corollary} Choose any orthonormal basis $\{v_j\}$
of $T_mM, m = \pi(q)$. Compute the determinant of the $\tfrac{n(n-1)}2 \times \tfrac{n(n-1)}2$ matrix
$$\det \Big(R(v_\alpha, v_\beta, v_i, v_j) - \widehat R(qv_\alpha, qv_\beta, qv_i, qv_j) \Big),$$
$$1 \leq \alpha < \beta \leq n \text{ are row indices,} \quad 1 \leq i < j \leq n \text{ are column indices,}$$
If this determinant is nonzero, then we have local controllability at $q$.
\end{corollary}

From this we obtain the following corollary.

\begin{corollary}
Define a function $\overline \varkappa_q$ on 2-dimensional planes $L$ in $D_q$ by the formula
$$\overline \varkappa_q(L) = \varkappa_{\pi(q)}(\pi_* L) - \widehat \varkappa_{\widehat \pi (q)} (\widehat \pi_* L),$$
where $\varkappa_m$ and $\widehat \varkappa_{\widehat m}$ denotes the respective sectional curvatures of $M$ and $\widehat M$ at the indicated points. Then
\begin{itemize}
\item[(a)] $\dim \calO_{q_0} = n$ if an only if $\overline{\varkappa}_{q} \equiv 0$ for any $q \in \calO_{q_0}$.
\item[(b)] If $\overline \varkappa_q > 0$ or $\overline \varkappa_q < 0$, then $D$ is bracket generating of step 3 at $q$.
\end{itemize}
\end{corollary}

\begin{proof}
If $\overline \varkappa_{q} \equiv 0$, then $\overline R_{q}$ is 0 also. Similarly, if $\overline \varkappa_p > 0$ (resp. $\overline \varkappa_p < 0$) for every $L$, then $-\overline R_q$ (resp. $ \overline R_q$) will be an inner product on $\bigwedge^2 D_q$.

To see this, for the case $\overline \varkappa_q > 0$, we only need to show that $\overline R_q(\overline \xi, \overline \xi) < 0$, whenever $\xi \in \bigwedge^2 D_q$ is nonzero. 
Pick an orthonormal basis basis $\overline v_1, \dots, \overline v_n$ of $D_q$. Write
$$L_{ij} := \spn \{ \overline v_i, \overline v_j\}.$$
In this basis, we have that if $\overline \xi = \sum_{1 \leq i < j \leq n} a_{ij} \overline v_i \wedge \overline v_j$ is nonzero, then
$$\overline R_q ( \overline \xi, \overline \xi) = -\sum_{1 \leq i < j \leq n} a_{ij}^2 \overline{\varkappa}_q (L_{ij} ) > 0.$$
The case $\overline \varkappa_q < 0$ is treated similarly.
\end{proof}

\begin{remark}
All the conditions stated here, are sufficient conditions for local controllability. However, if they hold in all points, they will naturally be sufficient conditions for complete controllability.
\end{remark}

\subsection{Examples}
\begin{example}
We start with two known examples, to verify our results and demonstrate their effectiveness of obtaining information on controllability. 
\begin{itemize} 
\item[(a)] If $M$ is a sphere of radius $r$ and $\widehat M= \real^n$ is the $n$ dimensional Euclidean space, then $M$ has constant sectional curvature $\frac{1}{r^2}$, while $\widehat M$ has constant sectional curvature 0. It follows that $\overline \varkappa_q \equiv \frac{1}{r^2}$ for any $q \in Q$. Hence $D$ is bracket generating at all points, and the system is completely controllable.
\item[(b)] If $M$ and $\widehat M$ are the spheres with respective radii $r$ and $\widehat r$, then
$$\overline \varkappa_q \equiv \frac{1}{r^2} - \frac{1}{\widehat r^2},$$
for any $q \in Q$. Hence the system is completely controllable if and only if $r \neq \widehat r$. When $r = \widehat r$, $D$ becomes an involutive distribution.
\end{itemize}
To compare, see \cite{GGMS,Zimm} for a former proof of the controllability of (a), and \cite{JZ} for a treatment of the example in (b).
\end{example}

\begin{example}
More generally, if $M$ is any manifold with only strictly positive or strictly negative sectional curvature, rolling on $n$-dimensional Euclidean space, then this system is completely controllable (we will later show that this only needs to hold in one point of $M$).
\end{example}

\begin{example}
Let $M = S^2 \times \real$ be the subset the Euclidian space $\real^4$,
$$\left\{(x_0, x_1, x_2, x_4) \in \real^4 \, : \, x_0^2 + x_1^2 + x_2^2 =1 \right\}.$$
Define a local section on the subset $U = \{( x_0, x_1, x_2, x_3) \in M \, : \, x_2 > 0 \},$
with the orthonormal vector fields
$$e_1 = - \sqrt{x_1^2 + x_2^2} \left(-\partial_{x_0} + \frac{x_0}{x_1^2 + x_2^2} \left(x_1 \partial_{x_1}
+ x_2 \partial_{x_2} \right)\right),$$
$$e_2 = \frac{x_2}{\sqrt{x_1^2 + x_2^2}}\left(-\partial_{x_1} + \frac{x_1}{x_2}\partial_{x_2} \right),
\qquad e_3 = \partial_{x_3}.$$
\begin{itemize}
\item[(a)] Let us first consider $M$ rolling on $\real^4$. The rolling distribution can locally be describes by
$$D^1 = \spn \{\overline e_1, \overline e_2, \overline e_3 \}.$$
$$\overline e_1 = e_1 + qe_1, \qquad \overline e_2 = e_2 + qe_2 + \frac{x_0}{\sqrt{x_1^2 + x_2^2}} W_{12},
\qquad \overline e_3 = e_3 + qe_3.$$
$D$ is then of step 3 for any $q \in U$ and
$$D^2 = D^1 \oplus \spn\{ W_{12}\}, \qquad D^3 = D^2 \oplus \spn \{ qe_1, qe_2\}$$
Since $D^3$ is locally finitely generated, we know that $\dim \calO_{q} = 6$ for any $q \in U$ (and for symmetry reasons, every $q \in Q$).
\item[(b)] Let $M$ roll on a copy of itself. Consider the rotation matrix $(q_{ij}) = (\langle e_i, qe_j\rangle).$
Give $(q_{ij})$ the coordinates
$$(q_{ij}) = \begin{tiny} \left(\begin{array}{ccc} \cos \theta \cos \varphi & \sin \theta \cos \psi - \cos \theta \sin \varphi \sin \psi
& \sin \theta \sin \psi + \cos \theta \sin \varphi \cos \psi \\
- \sin \theta \cos \varphi & \cos \theta \cos \psi + \sin \theta \sin \varphi \cos \psi &
\cos \theta \sin \psi- \sin \theta \sin \varphi \cos \psi\\
- \sin \varphi & - \sin \psi \cos \varphi & \cos \varphi \cos \psi \end{array} \right) \end{tiny}.$$
The vector fields spanning $D$ are locally given by
$$\overline e_1 = e_1 + qe_1 - \frac{x_0\left(\sin \theta \cos \psi + \cos \theta \sin \varphi \sin \psi \right)}{\sqrt{x_1^2 + 
x_2^2}} V,$$
$$\overline e_2 = e_2 + qe_2 + \frac{x_0}{\sqrt{x_1^2 + x_2^2}} W_{12}
- \frac{x_0\left(\cos \theta \cos \psi + \sin \theta \sin \varphi \sin \psi \right)}{\sqrt{x_1^2 + 
x_2^2}} V,$$
$$\overline e_3 = e_2 + qe_2 + \frac{x_0 \cos \varphi \sin \psi}{\sqrt{x_1^2 + x_2^2}} V,$$
$$V: =\cos \varphi \cos \psi W_{12}- \cos \varphi \sin \psi W_{13} - \sin \varphi W_{23}.$$
The matrix $-\Big(R(\overline e_\alpha, \overline e_\beta, \overline e_i, \overline e_j)\Big), \, i <j, \alpha < \beta$ is then given by
$$\begin{tiny} \left(\begin{array}{ccc}
1 - \cos^2 \varphi \cos^2 \psi & - \cos^2 \varphi \sin \psi \cos \psi & \cos \varphi \sin \varphi \cos \psi \\
- \cos^2 \varphi \sin \psi \cos \psi & - \cos^2 \varphi \sin^2 \psi & \cos \varphi \sin \varphi \sin \psi \\
\cos \varphi \sin \varphi \cos \psi & \cos \varphi \sin \varphi \sin \psi & - \sin^2 \varphi
 \end{array}\right) \end{tiny}.$$
 It is easily checked that this matrix has rank 2, except when $\sin \varphi = \sin \psi = 0$.
 Restricted to the subset of $Q$ where the latter equation hold, that is, the configurations where
the two copies of the line
$$\lambda  = \{(0,0,0,x_3) \in M \},$$
lie tangent to each other, $D$ is involutiove and the orbits are 3 dimensional. 
 
On the other points, we have that
$$D^2 = D^1 \oplus \spn \{ W_{12}, V \}, \qquad D^3 = D^2 \oplus \spn \{ qe_1, qe_2, qe_3\}.$$
so the orbits have dimension $8$, or codimension 1.

This example illustrates that if we are rolling manifolds of dimension higher than two, the dimension of $\mathcal O_q$ does not only depend on the connecting pair of points $\overline \pi(q) =(m, \widehat m)$.
\end{itemize}
\end{example}

\section{Particular cases}
We present some special results for when the manifolds involved in the rolling are particular nice. We will first deal with locally symmetric spaces, then present some results for rolling on complete spaces. Remark that all of these results are applicable to the case of rolling on $\real^n$, since this is both locally symmetric and complete.

\subsection{Locally symmetric spaces} \label{seclocsym}
Recall the definition of $\Omega$ from \eqref{OmegaDef}.
\begin{proposition}
Let $M$ be locally symmetric and let $\widehat M$ be flat ($\widehat R \equiv 0$). Then $D$ is at most of step 3. $D$ is bracket generating at $q \in Q$ if and only if
$$\Omega|_{\pi(q)}: \bigwedge^2 D_q \to \so(n),$$
is a linear isomorphism. 
\end{proposition}
\begin{proof}
Consider the bundle $\calD$, and let $X_j$ and $\widehat X_j$ be defined as in \eqref{basisEhresmann}.
Since $\widehat M$ is flat $[X_i + \widehat X_i, X_i + \widehat X_j] = [X_i, X_j].$ Then
$$[X_i, X_j] = -\scrW_{ij}(R), \qquad [X_k, [X_i, X_j]] = \sum_{s=1}^n \scrE_{skij}(R) X_s,$$
$$[X_l, [X_k, [X_i, X_j]]] = -\sum_{s=1}^n \scrE_{skij}(R) \scrW_{sl}(R) \in \calD^2,$$
Hence, $\calD_{(f,\widehat f)}^2 + \ker \varpi_{*(f,\widehat f)} = T_{(f,\widehat f)}\Big(F(M) \times \widehat F(M) \Big)$ only if
$$\spn \left\{ \scrW_{ij}(R) \big|_{(f, \widehat f)} \right\}_{i,j=1}^n = \spn \left\{ \sigma_{ij} \big|_{(f, \widehat f)} \right\}_{i,j=1}^n.$$
and this also implies that
$\spn \{ \sigma_{ij}, \sum_{s=1}^n \scrE_{skij}(R) X_s \} = \spn \{ \sigma_{ij}, f_j \},$
which gives us the desired result.
\end{proof}
When $\widehat M$ is not flat, the results become a little bit more complicated, and require us to introduce some notation. Let $R^l$ denote the $2l +2$-tensor defined by $R^1 = R$ and
$$R^l(Y_\alpha, Y_\beta, Y_{i_1}, Y_{i_2}, \dots, Y_{i_{2l-1}}, Y_{2l}) :=  R^{l-1}(R(Y_\alpha, Y_\beta) Y_{i_1}, Y_{i_2}, \dots, Y_{i_{2l-1}}, Y_{2l}).$$
\begin{lemma} \label{localsymmetriclemma}
If $\nabla R = 0$, then $\nabla R^l = 0$ for any $l \geq 1$.
\end{lemma}
\begin{proof}
We give the proof by induction. Assume that $\nabla R^k = 0$ for $1 \leq k < l.$
Let $m(t)$ be any smooth curve in $M$. Let $v_1(t), \dots, v_n(t)$ be parallel vector fields along $m(t)$.
Then
\begin{align*} &\frac{d}{dt} R^l(v_\alpha, v_\beta, v_{i_1}, v_{i_2}, \dots, v_{2l-1}, v_{2l}) \\
= &\frac{d}{dt} \sum_{s=1}^n R(v_\alpha, v_\beta, v_{i_l}, v_s) R^{l-1}(v_s, v_{i_2}, \dots, v_{2l-1}, v_{2l})
= 0. \end{align*}
Hence $\nabla R^l = 0$ also.
\end{proof}

We introduce a notation related to $R^l$, similar to what we did for $R$. Use $\widehat R^l$ for the analogues tensor on $\widehat M$, and write $\overline R^l$ for the tensor on $D$ defined by $\overline R^l = \pi^*(R^l) - \widehat \pi^*(\widehat R^l).$

\begin{proposition}
Let $M$ and $\widehat M$ both be locally symmetric. Then $D$ is bracket generating at $q$ if and only if
\begin{equation} \label{Rlunion} \bigcup_{l\geq1} \spn \bigg\{ \sum_{1 \leq \alpha < \beta \leq n} \overline R^l(\overline e_\alpha, \overline e_\beta, \overline e_{i_1}, \dots, \overline e_{i_{2l}} ) W_{\alpha \beta} \big|_q \bigg\} = \ker \overline \pi_{*}|_q.\end{equation}
\end{proposition}
\begin{proof}
We will look at the brackets of $\calD$.
From Lemma \ref{localsymmetriclemma}, we know that for any $l \geq 1$,
$$\big[X_{i_1} , \big[ X_{i_2}, \big[ \cdots \big[ X_{i_{2l-1}}, X_{i_{2l}}\big] \cdots \big] = (-1)^{l} \scrW_{i_1,\dots,i_{2l}}(R^l),$$
$$\big[X_{i_1} , \big[ X_{i_2}, \big[ \cdots \big[ X_{i_{2l}}, X_{i_{2l+1}}\big] \cdots \big] = (-1)^{l+1} \sum_{s=1}^n \scrE_{s, i_1,\dots, i_{2l+1}}(R^l) X_s. $$
Analogues relations hold for the brackets of $\widehat X_j$. Projecting the even brackets to $T_qQ$, we get the left hand side of \eqref{Rlunion}, which has to be equal to all of $\ker \overline \pi_{*}|_q$ in order for $D$ to be bracket generating. Conversely, if \eqref{Rlunion} holds, then the projection of the odd brackets will span $T_qQ$ together with $\ker \overline \pi_{*}|_q$ and $D_q$. \end{proof}

\subsection{Rolling on a complete manifold}
The fact that one of the manifolds is complete, makes it easier to give statements about complete controllability.
The reason for this, can be summed up in the following Lemma.
\begin{lemma} \label{completelemma}
Assume that $\widehat M$ is complete. Let $t \mapsto m(t)$ be any absolutely continuous curve in $M$ with domain $[0, \tau]$. Let $q_0 \in Q$ be any point with $\pi(q_0) = m(0)$.

Then there is a rolling $t \mapsto q(t)$ of $M$ on $\widehat M$, so that
$$q(0) = q_0, \qquad \pi \circ q(t) = m(t) \text{ for any } t \in [0,\tau].$$
\end{lemma}

\begin{proof}
If we assume first that both $M$ and $\widehat M$ are complete, then such a rolling $q(t)$ exist. The proof for this can be found in \cite[p. 386]{AS}. This proof is done for the case when $M$ and $\widehat M$ are two dimensional, but can, with simple modifications, be generalized to higher dimensions.

Assume now that $M$ is not complete. Let $f(t)$ be a lifting of $m(t)$ to a curve in $F(M)$ that is horizontal to the Ehresmann connection, i.e. each $f_j(t)$ is parallel along $m(t)$. Define the curve in $\real^n$ by,
$$\widetilde m(t) = \int_0^t f^{-1}(s)\big(\dot m(s)\big) ds.$$
Let $\widetilde f(t)$ be a lifting of $\widetilde m(t)$ to a curve $F(\real^n)$, so that each $\widetilde f_j(t)$ is parallel along $\widetilde m(t).$ Then, from Corollary \ref{stochasticroll}, $\widetilde q(t) = \varpi(f(t), \widetilde f(t))$ is a rolling of $M$ on $\real^n$.

Let $\widetilde q_0 := \widetilde q(0)$. Since both $\real^n$ and $\widehat M$ are complete, we know that there is a rolling $\widehat q(t)$ of $\real^n$ on $\widehat M$ along $\widetilde m(t)$, so that $\widehat q(0) = q_0 \circ \widetilde q_0^{-1}$. We can then obtain our desired rolling by defining $q(t) = \widehat q(t) \circ \widetilde q(t).$
\end{proof}

\begin{proposition}
Let the manifold $\widehat M$ be complete. Assume that there is a point $m \in M$, so that
$D_q$ is bracket generating for every point $q \in \pi^{-1}(m)$. Then the system is completely controllable. 
\end{proposition}

\begin{proof}
Let $q_0$ be any element in $Q$. From Lemma \ref{completelemma} we know that there is a rolling $q(t)$ from $q_0$ to some point $q_1 \in \pi^{-1}(m)$. Hence $\calO_{q_0} = \calO_{q_1}$. But since $D$ is bracket generating in $q_1$, $\calO_{q_1}$ is an open submanifold. Since $q_0$ was arbitrary, we have local controllability at every point,  so $\calO_{q_0} = Q$ for any $q_0 \in Q$.
\end{proof}

\begin{corollary}
Let $\widehat M$ be a manifold that is both complete and flat. Assume that there is a point $m \in M$, so that for some (and hence any) orthonormal basis $\{v_j\}_{j=1}^n$ of $T_mM$,
$$\det\left( R(v_\alpha, v_\beta, v_i, v_j) \right) \neq 0.$$
$$1\leq \alpha < \beta \leq 1 \text{ are row indices}, \quad 1\leq i < j \leq 1 \text{ are column indices}.$$
Then the system is completely controllable.
\end{corollary}

\begin{example}
Let $M$ be the surface in $\real^3$, defined by
$$M = \{(x_1,x_2,x_3) \in \real^3 \, : \, \sqrt{x_2^2 + x_3^2} = 1-f(x_1), |x_1| < \tfrac{3}{2} \},$$
where
$$f(x_1) = \left\{\begin{array}{ll} 0 & \text{if } |x_1| \leq 1 \\
e^{-\tfrac{1}{(|x_1|-1)^2}} & \text{if } 1< |x_1| < \frac{3}{2}
 \end{array} \right. .$$
Define the following orthonormal basis on $M$,
$$e_1 = \frac{1}{\sqrt{1 + f'(x_1)^2}} \left(\partial_{x_1} - \frac{f'(x_1)}{1 - f(x_1)}
(x_2 \partial_{x_2} + x_3 \partial_{x_3}) \right),$$
$$e_2 = \frac{1}{1 - f(x_1)}(-x_3 \partial_{x_2} + x_2 \partial_{x_3}).$$
All Christoffel symbols are determined by
$$\Gamma_{12}^1 = \langle e_1, \nabla_{e_1} e_2 \rangle = 0, \qquad
\Gamma_{22}^1 = \langle e_1, \nabla_{e_2} e_2 \rangle = \frac{f'(x_1)}{(1 - f(x_1))\sqrt{1 + f'(x_1)^2}}.$$
and from this we can compute the Gaussian curvature
$$\varkappa(x) = \frac{f''(x_1)}{(1+ f'(x_1)^2)^2 (1 - f(x_1))}.$$
Inserting the value of $f(x)$ we obtain that $\kappa(x) = 0$ for $|x_1| \leq 1$, but strictly positive for $1< |x_1| < \tfrac{3}{2}$.

It follows that , if we roll $M$ on $\real^2$, the system is completely controllable. Observe that in this case $$\Lie D = \spn \left\{e_1 + q e_1, e_2 + qe_2 + \Gamma_{22}^1 W_{12}, f(x_1) W_{12}, f(x_1) qe_1, f(x_1) qe_2\right\},$$
fails to be locally finitely generated around points fulfilling $|x_1| =1$.
\end{example}

\section{Further generalization of rolling without twisting or slipping}
Up until now, we have only been concerned with rolling two Riemannian manifolds on each other without twisting or slipping. The definition can easily be generalized to manifolds with an affine connection. We introduce the generalization here.

Let $M$ and $\widehat M$ be two connected manifolds, with respective affine connections $\nabla$ and $\widehat \nabla$. Then a rolling without twisting or slipping can be seen as an absolutely continuous curve $q(t)$ into the manifold
$$\mathcal Q = \left\{ q \in \GL(T_m, T_{\widehat m} \widehat M) \, : m \in M, \widehat m \in \widehat M \right\}.$$
satisfying (no slip condition) and (no twist condition) from section \ref{definition}.

Reexamining the proofs, it turns out that the description of many of the results we had for rolling related to the Levi-Civita connection, generalizes to general connections. We will describe this here briefly.

Define the tautological one-form on $\calF(M)$ as the $\real^n$-valued one-form obtained by
$$\theta(v) = f^{-1}(v), \qquad v \in T_f\calF(M).$$
Define $\widehat \theta$ similarly, while $\omega$ and $\widehat \omega$ are defined in terms of the connection. We will also choose Riemannian structures on $M$ and $\widehat M$ in order to present the results in a similar way and make them easier to compare, but these do not need to be compatible with the connections.

The rolling distribution $D$, will still be an $n$ dimensional distribution, and the relation in Theorem \ref{mathcalDtoD} is still valid. $D$ is locally spanned by vector fields
$$\overline e_j = e_j + qe_j + \sum_{\alpha, \beta =1}^n \left(\left\langle e_\alpha, \nabla_{e_j} e_\beta \right\rangle
- \left\langle qe_\alpha, \widehat \nabla_{qe_j} q e_\beta\right\rangle \right) E_{ij}^{\ell}.$$
where $E_{ij}^{\ell} = \sum_{s=1} q_{ri} \frac{\partial}{\partial q_{rj}}.$

The study of controllability becomes harder, since torsion may appear in the equations \eqref{Cartanequations}. However, this is actually the only complication. Let $\epsilon_{\alpha\beta} \in \gl(n,\real)$ be the matrix with only $1$ in at entry $\alpha\beta$ and zero otherwise. Define $X_j$ as in equation \eqref{basisEhresmann}. If we then modify the definition of $\scrW_{i_1,\dots, i_l}(T)$ to
$$\scrW_{i_1,\dots, i_l}(T) = \sum_{\alpha,\beta=1}^n \scrE_{\alpha,\beta,i_1,\dots, i_l}(T) \sigma(\epsilon_{\alpha\beta}),$$
now defined for any tensor $T$, then Lemma \ref{WErelations} still holds for any connection. Since all of our results follow from Lemma \ref{WErelations} and \ref{Xbracketlemma}, it follows that our results also holds for any pair of manifolds with torsion free connections.


\begin{thebibliography}{99}

\bibitem{A} {\sc A.~Agrachev}, {\it Rolling balls and octonions}. Proc. Steklov Inst. Math. {\bf 258} (2007), 13--22.

\bibitem{AS} {\sc A.~Agrachev, Y.~Sachkov}, Control Theory from the Geometric Viewpoint, Springer, 2004.

\bibitem{AmSi53} W.~Ambrose, I.~M.~Singer,
{\it A theorem on holonomy.}
Trans. Amer. Math. Soc. {\bf 75}, (1953). 428--43. 

\bibitem{BH} {\sc R. Bryant, L. Hsu}, {\it Rigidity of integral curves of rank $2$ distributions}. Invent. Math. {\bf 114} (1993), no. 2, 435--461.

\bibitem{BM} {\sc G.~Bor, R.~Montgomery}, {\it $G_2$ and the rolling distribution}. L'Ens. Math. (2) {\bf 55} (2009), 157--196.

\bibitem{Chitour} {\sc Y. Chitour, P. Kokkonen} {\it Rolling Manifolds: Intrinsic Formulation and Controllability}
arXiv:1011.2925



\bibitem{Chow} {\sc W.~L.~Chow}, {\it Uber Systeme von linearen partiellen Differentialgleichungen erster Ordnung}, Math. Ann. {\bf 117} (1939), 98--105.

\bibitem{GGMS} {\sc M.~Godoy~M., E.~Grong, I.~Markina, F.~Silva~Leite} {\it An intrinsic formulation
of the Rolling Manifold problem.} arXiv:1008.1856

\bibitem{Hermann} {\sc R.~Hermann}, {\it On the accessibility problem in control theory}, Internat. Sympos. Nonlinear Differential Equations and Nonlinear Mechanics, Academic Press, New York (1963), 325--332.

\bibitem{Hsu} {\sc E.~P.~Hsu} {\it Stochastic Analysis on Manifolds} American Mathematical Society 2002

\bibitem{HS} {\sc K.~H{\"u}per, F.~Silva Leite}, {\it On the geometry of rolling and interpolation curves on {$S\sp n$}, {${\rm SO}\sb n$}, and {G}rassmann manifolds}. J. Dyn. Control Syst. {\bf 13} (2007), no. 4, 467--502.

\bibitem{HKS} {\sc K.~H{\"u}per, M. Kleinsteuber, F.~Silva Leite}, {\it Rolling {S}tiefel manifolds}. Internat. J. Systems Sci. {\bf 39} (2008), no. 9, 881--887.

\bibitem{IW} {\sc N. Ikeda, S. Watanabe} {\it Stochastic Differential Equations and Diffusion Processes}
North-Holland Publishing Company (1981)

\bibitem{Jost} {\sc J. Jost} {\it Riemannian Geometry and Geometric Analysis}, Springer-Verlag Berlin Heidelberg,2005

\bibitem{Jur} {\sc V. Jurdjevic} {\it Geometric Control Theory} Cambridge University Press 1997

\bibitem{JZ} {\sc V. Jurdjevic, J.~A.~Zimmerman}  {\it Rolling sphere problems on spaces of constant curvature}
Math. Proc. Cambridge Philos. Soc. {\bf 144} (2008), no. 3, 729--747


\bibitem{MB}{\sc A. Marigo, A. Bicchi}{\it Rolling Bodies with Regular Surface: Controllability
Theory and Applications} IEEE Transactions on Automatic Control {\bf 45} (2000), no. 9, p 1586--1599.

\bibitem{Mon02} R.~Montgomery,
{\it A tour of subriemannian geometries, their geodesics and applications}, Mathematical Surveys and Monographs, {\bf 91}. American Mathematical Society, Providence, RI, 2002. 



\bibitem{Rashevsky} {\sc P.~K.~Rashevski{\u\i}}, {\it About connecting two points of complete nonholonomic space by admissible curve}, Uch. Zapiski Ped. Inst. K.~Liebknecht {\bf 2} (1938), 83--94.


\bibitem{Sharpe} {\sc R. W. Sharpe}, Differential geometry. GTM, 166. Springer-Verlag, New York, 1997. 

\bibitem{Shen} {\sc Y.~Shen, K. H\"uper, F.~Silva Leite} {\it Smooth Interpolation of Orientation by Rolling and Wrapping for Robot Motion Planning}
Proceedings of the 2006 IEEE International Conference on Robotics and Automation (ICRA2006), Orlando, USA, May 2006.


\bibitem{Suss} {\sc H.~J.~Sussmann}, {\it Orbits of families of vector fields and integrability of distributions}, Internat. Sympos. Trans. Amer. Math. Soc. {\bf 180} (1973), 171--188.

\bibitem{Zimm}{J.~A.~Zimmerman}, {\it Optimal control of the sphere $S^n$ rolling on $E^n$},  Math. Control Signals Systems {\bf 17} (2005), no. 1, 14--37.

\end{thebibliography}
\end{document}